\documentclass[11pt,leqno]{amsart}

\usepackage{enumerate}

\usepackage{mathrsfs}

\usepackage{amsthm}

\usepackage[english]{babel}

\usepackage{amsmath}

\usepackage{amssymb}

\usepackage{xypic}

\usepackage[all]{xy}

\usepackage{graphicx}
\usepackage{amsfonts}

\usepackage{pstricks,pst-plot,epsfig}
\usepackage{enumerate}

\usepackage[T1]{fontenc}

\numberwithin{equation}{section}

\newcommand{\C}{\mathcal{C}}

\newcommand{\D}{\mathcal{D}}

\newcommand{\K}{\mathcal{K}}

\newcommand{\T}{\mathcal{T}}

\newcommand{\U}{\mathcal{U}}

\renewcommand{\mod}{\mathrm{Mod}}

\newcommand{\cov}{\mathrm{Cov}}

\newcommand{\coh}{\mathrm{Coh}}

\newcommand{\op}{\mathrm{Op}}

\newcommand{\rc}{\mathbb{R}\textrm{-}\mathrm{c}}

\newcommand{\CC}{\mathbb{C}}

\newcommand{\R}{\mathbb{R}}

\newcommand{\N}{\mathbb{N}}

\newcommand{\OO}{\mathcal{O}}

\newcommand{\osi}{\stackrel{\sim}{\gets}}

\newcommand{\iso}{\stackrel{\sim}{\to}}

\newcommand{\dbt}{\mathcal{D}\mathit{b}^t}

\newcommand{\db}{\mathcal{D}\mathit{b}}

\newcommand{\ot}{\mathcal{O}^t}

\newcommand{\tho}{\mathit{T}\mathcal{H}\mathit{om}}

\newcommand{\OW}{\OO^\mathrm{w}}

\newcommand{\CW}{\C^{{\infty ,\mathrm{w}}}}
\newcommand{\wtens}{\overset{\mathrm{w}}{\otimes}}

\newcommand{\rh}{\mathit{R}\mathcal{H}\mathit{om}}

\newcommand{\ho}{\mathcal{H}\mathit{om}}

\newcommand{\Ho}{\mathrm{Hom}}

\newcommand{\Rh}{\mathrm{RHom}}

\newcommand{\id}{\mathrm{id}}

\newcommand{\imin}[1]{#1^{-1}}

\newcommand{\lind}[1]{\underset{#1}{\varinjlim}}

\newcommand{\Lind}{\underrightarrow{\lim}}  

\newcommand{\lpro}[1]{\underset{#1}{\varprojlim}}

\newcommand{\exs}[3]{0 \to {#1} \to {#2} \to {#3} \to 0}

\newcommand{\lexs}[3]{0 \to {#1} \to {#2} \to {#3}}

\newcommand{\dt}[3]{{#1} \to {#2} \to {#3} \stackrel{+}{\to}}

\newtheorem{teo}{Theorem}[section]

\newtheorem{cor}[teo]{Corollary}

\newtheorem{oss}[teo]{Remark}

\newtheorem{prop}[teo]{Proposition}

\newtheorem{lem}[teo]{Lemma}

\newtheorem{es}[teo]{Example}

\newtheorem{nt}[teo]{Notations}

\usepackage{latexsym}

\begin{document}
\author{Teresa Monteiro Fernandes and Luca Prelli}
\title{Relative subanalytic sheaves}

\thanks{The research of T.Monteiro Fernandes was supported by
Funda\c c{\~a}o para a Ci{\^e}ncia e Tecnologia, PEst OE/MAT/UI0209/2011.}

\thanks{The research of L. Prelli was supported by Marie Curie grant PIEF-GA-2010-272021.}

\address{}

\address{Teresa Monteiro Fernandes and Luca Prelli\\ Centro de Matem\'atica e Aplica\c{c}\~{o}es Fundamentais e Departamento de Matem\' atica da FCUL, Complexo 2\\ 2 Avenida Prof. Gama Pinto, 1649-003, Lisboa
 Portugal\\
  \textit{tmf@ptmat.fc.ul.pt, lprelli@math.unipd.it}}

\address{}


\begin{abstract}
Given a real analytic manifold $Y$, denote by $Y_{sa}$ the associated subanalytic site. Now consider a product $Y=X\times S$. We construct the endofunctor $\mathcal{F}\mapsto \mathcal{F}^{S}$ on the category of sheaves on $Y_{sa}$ and study its properties. Roughly speaking, $\mathcal{F}^S$ is a sheaf on $X_{sa}\times S$. As an application, one can now define sheaves of functions on $Y$ which are tempered or Whitney in the relative sense, that is, only with respect to $X$.\end{abstract}
\maketitle

\tableofcontents

\section*{Introduction}

Let $Y$ be a real analytic manifold. The subanalytic sheaf $\db_Y^t$ of tempered distributions defined by Kashiwara-Schapira (\cite{KS3}) takes its origin in  Kashiwara's  functor $TH$ (\cite{Ka1}) as an essential tool to establish the Riemann-Hilbert correspondance between regular holonomic $\D$-modules and perverse sheaves.

Let  $Y=X\times S$, for some real analytic manifolds $X$ and $S$. In order to study relative perversity (\cite{TS}), it appears that a ``relative'' version of $\dbt_{X \times S}$ is required, i.e. a sheaf $\db^{t,S}_{X \times S}$ such that
$$
\Gamma(U \times V;\db^{t,S}_{X \times S}) \simeq \lpro {W \subset\subset V} \Gamma(U \times W;\dbt_{X \times S}).
$$
In other words, such a sheaf ``forgets'' the growth conditions on $S$.

Let $\mod(\CC_{(X \times S)_{sa}})$ 
be the category 
of subanalytic sheaves on $X \times S$. The aim of this note is to construct a functor  $(\cdot)^{S} : \mod(\CC_{(X \times S)_{sa}})  \to  \mod(\CC_{(X \times S)_{sa}})$ such that, given $F \in \mod(\CC_{(X \times S)_{sa}})$,
\begin{equation}\label{eq}
\Gamma(U \times V;F^S) \simeq \lpro {W \subset\subset V} \Gamma(U \times W;F),
\end{equation}
or, more generally, when $F$ is a bounded complex of subanalytic sheaves and $G$ (resp. $H$) is a bounded complex of $\R$-constructible sheaves on $X$ (resp. $S$), its derived version $(\cdot)^{RS}$ satisfying
\begin{equation}\label{eq}
\rh(G \boxtimes H,F^{RS})  \simeq \rh(\CC_X \boxtimes H,\imin\rho\rh(G \boxtimes \CC_S,F)),
\end{equation}
where $\rho: X \times S \to (X \times S)_{sa}$ is the natural functor of sites.\\

Recall that, by the definition of  $\T$-space introduced in \cite{KS3}  (cf also \cite{EP}), the usual subanalytic site $(X \times S)_{sa}$ can also be regarded as the site $(X \times S)_{\T}$ where $\T$ is the family of all relatively compact subanalytic open subsets. If we consider as $\T'$ the family  of finite unions of open relatively compact subsets of the form $U\times V$, with $U$ subanalytic in $X$ and $V$ subanalytic in $S$, then $X \times S$ becomes a $\T'$-space, and the associated site is the product of sites $X_{sa}\times S_{sa}$. One notes by $\eta$ the morphism of sites $(X\times S)_{sa}\to X_{sa}\times S_{sa}$, by $\rho$ the morphism of sites $X\times S\to (X\times S)_{sa}$ and by $\rho'$ the morphism of sites $X\times S\to X_{sa}\times S_{sa}$.

In this note, to any $\T$-sheaf $F$ (that is, a sheaf on the site associated to $\T$,  or a subanalytic sheaf) we associate canonically  a $\T'$-sheaf $F^{S, \sharp}$ which in some way forgets the dependance of $F$ on the subanalytic factor $S_{sa}$. We then define the relative sheaf $F^S$ as the inverse image by $\eta$ of the $\T'$-sheaf $F^{S, \sharp}$, thus obtaining a subanalytic sheaf on $(X\times S)_{sa}$. This construction leads to a left exact functor $(\cdot)^S$ from the abelian category of subanalytic sheaves on $X\times S$ into itself. Denoting by $(\cdot)^{RS}$ its right derived functor, we prove in Proposition \ref{lem 9} that $(\cdot)^{RS}$ satisfies,
 for $F \in D^b(\CC_{(X \times S)_{sa}})$, $G \in D^b_{\rc}(\CC_X)$ and $H \in D^b_{\rc}(\CC_S)$, natural isomorphisms
\begin{eqnarray*}
\imin \rho \rh(G \boxtimes H,F^{RS}) & \simeq  & \imin \rho \rh(G \boxtimes \rho_!H,F) \\
& \simeq & \rh(\CC_X \boxtimes H,\imin\rho\rh(G \boxtimes \CC_S,F)),
\end{eqnarray*}
In particular, when $G=\CC_X$ and $H=\CC_S$ we have $\imin\rho F \simeq \imin\rho F^{RS}\simeq \rho'^{-1} F^{RS,\sharp}$.\\

We then apply our construction to $\db_{X\times S}^t$  and obtain the subanalytic sheaf $\db_{X\times S}^{t, S}$ of relative tempered distributions with respect to a projection  $f:X\times S\to S$. As the denomination suggests, this is a sheaf on the subanalytic site $(X\times S)_{sa}$, whose sections on open subsets of the form $U\times V$ are  distributions  which extend to  $X\times V$.

The same procedure applies to  construct the subanalytic sheaves  $\C_X^{\infty, t,S}$ of relative tempered $\C^{\infty}$-functions and $\C_X^{\infty,\mathrm{w}, S}$ of relative Whitney $\C^{\infty}$-functions on $X_{sa}$.
 Proposition \ref{lem 9} shows that taking  inverse images on $X\times S$ for the usual topology, we recover respectively the classical sheaves of distributions and $\C^{\infty}$-functions forgetting the relative growth conditions.

When $X$ and $S$ are complex manifolds,  the classical procedure of taking the Dolbeault complex applies to our constructions thus allowing  to define the subanalytic (complexes) ${\mathcal{O}}^{t,S}_{X\times S}$ of relative tempered holomorphic functions and $\OO_{X\times S}^{\mathrm{w}, S}$ of relative Whitney holomorphic functions.\\




As the reader can naturally ask, our method applies only for products of
analytic manifolds (see Remark \ref{controesempio}). We conjecture that with a weaker
notion of subanalytic site as in \cite{GS}, a notion of relative sheaf can be
given for a general smooth function but it will not suit to the
applications we have in scope.

However, the tools we develop here, besides its own interest, are useful to the further understanding of the notion of relative perversity introduced in \cite{TS}.
\\

\section{Complements on subanalytic $\T$-sheaves}\label{S:1}

The results in this section rely on the notion of $\T$-topology. References for details  are made to \cite{KS3} and \cite{EP} from which we keep the notations.


Given a topological space $X$ and a family $\T$ of open subsets of $X$, one says that $X$ is a $\T$-space if $\T$ satisfies the following conditions:
\begin{enumerate}
\item{$\T$ is a basis of the topology of $X$ and $\emptyset\in\T$,}
\item{$\T$ is closed under finite unions and intersections,}
\item{for any $U\in\T$, $U$ has finitely many $\T$-connected components.}
\end{enumerate}
 To $\T$ one associates a Gro\-then\-dieck topology in the following way:
a family $\U=\{U_i\}_i$ in $\T$ is a covering of $U \in \T$ if it admits a finite subcover. One denotes by $X_{\T}$ the associated site and by  $\rho: X \to X_{\T}$  the natural morphism of sites.
There are well defined functors:
\begin{equation}\label{rho}
\xymatrix{\mod(\CC_X)
\ar@ <2pt> [r]^{\mspace{0mu}\rho_*} &
  \mod(\CC_{X_\T}) \ar@ <2pt> [l]^{\mspace{0mu}\imin \rho}. }
\end{equation}

Let us consider the category $\mod(\CC_X)$ of sheaves of
$\CC_X$-modules on $X$, and let us denote by $\K$ the subcategory whose
objects are the sheaves $\oplus_{i \in I} k_{U_i}$ with $I$
finite and $U_i \in \T$ for each $i$. Let $F \in \mod(\CC_X)$.
\begin{itemize}
\item[(i)] $F$ is $\T$-finite if there exists an epimorphism
$G \twoheadrightarrow F$ with $G \in \K$.
\item[(ii)] $F$ is $\T$-pseudo-coherent if for any morphism
$\psi:G \to F$ with $G \in \K$, $\ker \psi$ is $\T$-finite.
\item[(iii)] $F$ is $\T$-coherent if it is both $\T$-finite
and $\T$-pseudo-coherent.
\end{itemize}
Remark that (ii) is equivalent to the same condition with ``$G$ is $\T$-finite" instead of ``$G \in \K$". One denotes by
$\coh(\T)$ the full subcategory of $\mod(k_X)$ consisting of $\T$-coherent sheaves.  $\coh(\T)$ is additive and stable by kernels.

Moreover:
\begin{itemize}
\item{Let $W \in \T$ and let $\CC_{W_\T} \in \mod(\CC_{X_\T})$ be the constant sheaf on $W$. Then $\rho_*\CC_W \simeq \CC_{W_\T}$.}
\item{The functor $\rho_*$ is fully faithful. Moreover its restriction to $\coh(\T)$ is exact.}
\item{A sheaf $F \in \mod(\CC_{X_\T})$ can be seen as a filtrant inductive limit $\lind i \rho_*F_i$ with $F_i \in \coh(\T)$.}
\item{The functors $\Ho(G,\cdot)$, $\ho(G,\cdot)$, with $G \in \coh(\T)$, commute with filtrant $\Lind$.}
\end{itemize}
Finally, recall (cf \cite{EP}) that $F \in \mod(\CC_{X_\T})$ is $\T$-flabby if the restriction morphism $\Gamma(X;F) \to \Gamma(W;F)$ is surjective for each $W \in \T$. $\T$-flabby objects are $\Ho(G,\cdot)$-acyclic for each $G \in \coh(\T)$. \\
Given a real analytic manifold $Y$, let $\op^c(Y_{sa})$ (resp. $\op(Y_{sa})$) denote the family of subanalytic relatively compact open subsets in $Y$ (resp. the family of subanalytic  open subsets in $Y$). Let $Y_{sa}$ denote the associated subanalytic site introduced in \cite{KS3}. The site $Y_{sa}$ is the site $Y_\T$ associated to the family $\T=\op^c(Y_{sa})$  (that is, $Y$ is a $\T$-space and the associated site $Y_{\T}$ coincides with $Y_{sa}$). Accordingly we shall still denote by $\rho$ the natural functor of sites $\rho:Y \to Y_{sa}$  associated to the inclusion $\op(Y_{sa}) \subset \op(Y)$   (without reference to $Y$ unless otherwise specified),  as well as the associated functors $\rho_{*}, \rho^{-1}\,, \rho_!$ introduced in \cite{KS3} (cf also \cite{P}).

 Let us recall the following  facts:
\begin{itemize}
\item{The functor $\rho_!$ is right adjoint to $\imin\rho$. It is fully faithful and exact. Given $F \in \mod(\CC_Y)$, $\rho_!F$ is the sheaf associated to the presheaf $\op(Y_{sa}) \ni U \mapsto F(\overline{U})$.}
\item{The category $\mod_{\rc}(\CC_Y)$ of $\R$-constructible sheaves is $\rho_*$-acyclic (cf \cite{P}).}
\end{itemize}

Let now be given a  sub-family $\T'\subset \op^c(Y_{sa})$   such that $Y$ is still a $\T'$-space.

\begin{itemize}
\item{Denoting by  $Y_{\T'}$  the site   associated to the family $\T'$,  we shall also denote by $\rho':Y\to Y_{\T'}$ the natural functor of sites. A sheaf $F \in \mod(\CC_{Y_{\T'}})$ can be seen as a filtrant inductive limit $\lind i \rho'_*F_i$ with $F_i \in \coh(\T')$.}
\item{We shall denote by $\eta$ the natural functor of sites $Y_{sa}\to Y_{\T'}$.}
\end{itemize}
 One obtains a commutative diagram of sites
\begin{equation} \label{diagram of sites}
\xymatrix{
Y \ar[d]_{\rho'} \ar[r]^\rho & Y_{sa} \ar[dl]^\eta \\
Y_{\T'} &
}
\end{equation}


\begin{oss} One could also consider the site defined by the family of locally finite unions of elements of $\T$ (in the case $\T=\op^c(Y_{sa})$ these are all subanalytic open subsets) and locally finite coverings and make the same construction  using the family $\T'$. Since the associated categories of sheaves are respectively isomorphic to $\mod(\CC_{Y_{\T}})$ and $\mod(\CC_{Y_{\T'}})$ (see Remark 6.3.6 of \cite{KS3})  we will still denote by $Y_\T$ (resp. $Y_{\T'}$) the associated site.
\end{oss}

Let $F$ be a sheaf on $Y_{\T'}$. One defines the (separated) presheaf $\eta^\dagger F$ on $Y_{sa}$ by setting, for $W \in \op(Y_{sa})$,
$$
\eta^\dagger F(W)= \lind {W \subset W'}F(W')
$$
with $W' \in \op(Y_{\T'})$. Let $\imin\eta F$ be the associated sheaf.

\begin{lem}\label{lem: structure of inverse image} Let $F \simeq \lind i\rho'_*F_i \in \mod(\CC_{Y_{\T'}})$ with $F_i \in \coh(\T')$. Then $\imin\eta F \simeq \lind i\rho_*F_i$.
\end{lem}
\begin{proof} Since the functor of inverse image commutes with $\Lind$ it is enough to check that $\imin\eta\rho'_*F' \simeq \rho_*F'$ with $F' \in \coh(\T')$.

Since (cf \cite{KS3}, Chapter 6) $\coh(\T')$ is an abelian subcategory of $\coh(\T)$ and $\rho_*$ (resp. $\rho'_*$) is exact on $\coh(\T)$ (resp. $\coh(\T')$), we may reduce to the case $F'=\CC_W$, $W \in \T'$.

Let $\CC_{W_\T}$ (resp. $\CC_{W_{\T'}}$) be the constant sheaf on $Y_{sa}$ (resp. $Y_{\T'}$). Then, by Proposition 6.3.1 of \cite{KS3}, (cf also \cite{P})$$
\imin\eta \rho'_*\CC_W \simeq \imin\eta \CC_{W_{\T'}} \simeq \CC_{W_\T} \simeq \rho_*\CC_W.
$$
\end{proof}

\begin{lem}\label{lem: section on product0}
Let $F\in  \mod(\CC_{Y_{\T'}})$. Then, for any  $W \in \op(Y_{\T'})$, $$\Gamma(W;\eta^{-1}F)\simeq \Gamma(W; F).$$
\end{lem}
\begin{proof}
 We may write $F \simeq \lind i\rho'_*F_i$ with $F_i \in \coh(\T')$ and by Lemma \ref{lem: structure of inverse image} we have $\imin\eta F \simeq \lind i\rho_*F_i$.

Let us first assume that $W \in \op(Y_{\T'})$ is relatively compact. Then
\begin{eqnarray*}
\Gamma(W;\lind i\rho_*F_i) & \simeq & \lind i\Gamma(W;\rho_*F_i) \\
& \simeq & \lind i\Gamma(W;F_i) \\
& \simeq & \lind i\Gamma(W;\rho'_*F_i) \\
& \simeq & \Gamma(W;\lind i\rho'_*F_i).
\end{eqnarray*}

 Let us consider now an arbitrary $W$.  Then we have $W=\bigcup_nW_n$, with $W_n=U_n  \cap W$, where $\{U_n\}_{n \in \N}$ belongs to $\cov(Y_{\T'})$ and satisfies $U_n\subset\subset U_{n+1}$.
Therefore:
\begin{eqnarray*}
\Gamma(W;\lind i\rho_*F_i) & \simeq & \underset{n}{\varprojlim}\Gamma(W_n; \lind i {\rho_*F_i}) \\
& \simeq & \underset{n}{\varprojlim}\Gamma(W_n; \lind i {\rho'_*F_i}) \\
& \simeq & \Gamma(W;\lind i\rho'_*F_i).
\end{eqnarray*}

\end{proof}

The two following results are straightforward consequences of Lemma \ref{lem: section on product0}:

\begin{cor} The adjunction morphism $\id \to \eta_*\imin\eta$ is an isomorphism. In particular, the functor $\imin\eta:\mod(\CC_{Y_{\T'}}) \to \mod(\CC_{Y_{sa}})$ is fully faithful.
\end{cor}

\begin{cor} Let $W \in \T'$ and let $\CC_{W_\T}$ (resp. $\CC_{W_{\T'}}$) be the constant sheaf on $Y_{sa}$ (resp. $Y_{\T'}$). Then
$
\CC_{W_{\T'}} \simeq \eta_*\CC_{W_\T}.
$
\end{cor}

Let $\mathcal{I}$ be the subcategory of $\mod(\CC_Y)$ consisting of finite sums $\oplus_i\CC_{W_i}$ with $W_i \in \T'$ connected.

\begin{lem} \label{lem: finite res} Let $F,G \in \mathcal{I}$. Then, given $\varphi:F \to G$, we have $\ker\varphi \in \mathcal{I}$.
\end{lem}
\begin{proof} We have $F=\oplus_{i=1}^l\CC_{W_i}$, $G=\oplus_{j=1}^k\CC_{W'_j}$. Composing with the projection $p_j,\, j=1,..., k$ on each factor of $G$, $\ker\varphi$ will be the intersection of the $\ker p_j\circ \varphi$ so that, if each one has the desired form, the same will happen to their intersection.  Therefore it  is sufficient to assume $k=1$, let us say, $G=\CC_{W}$. A morphism $\varphi:F \to G$ is then defined by a sequence $v=(v_{1},\dots,v_{l})$, where $v_i$ is the image by $\varphi$ of the section of $\CC_{W_i}$ defined by $1$ on $W_i$, so $v_{i}=0$ if $W_i\not\subset W$. More precisely, if $s=(s_1,...,s_l)$ is a germ  of $F$ in $y$, we have $\varphi(s_1,..., s_l)=\sum_{i=1}^l{v_i}_{y}{s_i}$. So, given $s=(s_1, ..., s_l)\in\ker\varphi$,  if, for  a given $i$, we have ${v_i}_{y} s_i\neq 0$, then $s$ defines a germ of  $H_i=:\oplus_{i' \neq i}\CC_{W_{i'}\cap W_{i}}$ in $y$.


Accordingly, $\ker\varphi\simeq\oplus_{i=1}^l H_i$.
\end{proof}

Therefore, according to the definition of $\coh(\T')$ and to Lemma \ref{lem: finite res}, any  $F\in\coh(\T')$ admits a finite resolution
$$
I^{\bullet}:=0 \to I_1 \to \cdots \to I_n \to F \to 0
$$
consisting of objects belonging to $\mathcal{I}$.

\begin{lem} \label{lem: coh acyclic} Let us suppose that, for any $U\in\T'$, $\CC_U$ is  $\rho'_*$-acyclic. Then:
\begin{enumerate}

\item{for any
$F \in \coh(\T')$, $F$ is $\rho'_*$-acyclic or, equivalently, $\rho_*F$ is $\eta_*$-acyclic.}
\item{ Let $F \in D^b(\CC_{Y_{\T'}})$. Then $R\Gamma(W;\imin\eta F) \simeq R\Gamma(W;F)$ for each $ W\in \op(Y_{\T'})$. }
\end{enumerate}

\end{lem}
\begin{proof}
$(1)$ The equivalence of the two assertion follows from the fact that $R\rho'_*=R\eta_* \circ R\rho_*$ and $\R$-constructible sheaves (and hence $\T'$-coherent sheaves) are $\rho_*$-acyclic.

Note that the assumption entails that any quotient  $I_1/I_2$ of elements of $\mathcal{I}$  is $\rho_*'$-acyclic, so $F$ is $\rho_*'$-acyclic.


(2) Let $F \in D^b(\CC_{Y_{\T'}})$. By d\'evissage, we may reduce to $F \in \mod(\CC_{Y_{\T'}})$ and we can write $\imin\eta F \simeq \lind i\rho_*F_i$, with $F_i\in\coh({\T'})$.  There exists (see \cite{KS4},  Corollary 9.6.7) an inductive system of injective resolutions $I^\bullet_i$ of $F_i$. By $(1)$ (resp. Lemma 2.1.1 of  \cite{P}) $F_i$ is $\rho'_*$-acyclic (resp. $\rho_*$-acyclic),  hence $\rho'_*I^\bullet_i$ (resp. $\rho_*I^\bullet_i$) is an injective resolution of $\rho'_*F_i$ (resp. $\rho_*F_i$). Then, with the notations of \cite{EP}, $\lind
i\rho'_*I_i^\bullet$ (resp. $\lind
i\rho_*I_i^\bullet$) is a $\T'$-flabby (resp. $\T$-flabby) resolution of $\lind i \rho'_*F_i$ (resp. $\lind i \rho_*F_i$) and hence $\Gamma(W;\cdot)$-acyclic. We
have
\begin{eqnarray*}
R\Gamma(W;\lind i\rho_*F_i) & \simeq & \Gamma(W;\lind i\rho_*I^\bullet_i) \\
& \simeq & \Gamma(W;\lind i\rho'_*I^\bullet_i) \\
& \simeq & R\Gamma(W;\lind i\rho'_*F_i),
\end{eqnarray*}
where the second isomorphism follows from Lemma \ref{lem: section on product0}. \end{proof}

The following two results are straightforward consequences of Lemma \ref{lem: coh acyclic}:

\begin{cor} Under the assumption of Lemma \ref{lem: coh acyclic}, the adjunction morphism $\id \to R\eta_*\imin\eta$ is an isomorphism. In particular, the functor $\imin\eta:D^b(\CC_{Y_{\T'}}) \to D^b(\CC_{Y_{sa}})$ is fully      faithful.
\end{cor}
Note that  Remark \ref{oss:T} in next section provides an example showing that  the converse $\imin\eta R\eta_*\to\id $ is not in general an isomorphism.
\begin{cor} Assume the conditions of Lemma \ref{lem: coh acyclic}.
Let $W \in \T'$ and let $\CC_{W_\T}$ (resp. $\CC_{W_{\T'}}$) be the constant sheaf on $Y_{sa}$ (resp. $Y_{\T'}$). Then
$
\CC_{W_{\T'}} \simeq R\eta_*\CC_{W_\T}.
$
\end{cor}

As a consequence of Lemma \ref{lem: structure of inverse image} we obtain:

\begin{cor}\label{C:25}
Assume the conditions of Lemma \ref{lem: coh acyclic}.
Let $F \in D^b(\coh(\T'))$. Then
$\imin\eta R\rho'_*F \iso R\rho_*F.$
\end{cor}
\begin{proof}
We have the chain of isomorphisms
$$
\imin\eta R\rho'_* F  \simeq  \imin\eta \rho'_* F 
 \simeq  \rho_*F 
 \simeq  R\rho_*F,
$$
where the first and the last isomorphisms follow since $\rho_*'$ and $\rho_*$ are acyclic on $\coh(\T')$ and the second isomorphism is part of the proof of Lemma \ref{lem: structure of inverse image}.
\end{proof}

\section{The case of a product}
Hereafter we will consider the case where $Y$ is a product $X\times S$ of real analytic manifolds.
On $X\times S$ it is natural to consider the family $\T'$ consisting of finite unions of open relatively compact subsets of the form $U\times V$ which makes  $X\times S$  a $\T'$-space. The  associated  site $Y_{\T'}$ is nothing more than the product of sites $X_{sa}\times S_{sa}$.
 Let $p_1:X \times S \to X$ and $p_2: X \times S \to S$ be the projections.

Note that $W \in \op(X_{sa}\times S_{sa})$ is a locally finite union of relatively compact subanalytic open subsets of the form $U \times V$, $U \in \op(X_{sa})$, $V \in \op(S_{sa})$. Accordingly to Section 1, we denote by  $\eta:(X \times S)_{sa} \to X_{sa} \times S_{sa}$ the natural functor of sites associated to the inclusion $\op(X_{sa} \times S_{sa}) \hookrightarrow \op((X \times S)_{sa})$.


We shall need the following result:
\begin{lem}\label{L:26}
Let $F,G$ be objects of $D^b(\coh({\T'}))$. Then $\rh(F,G)$ is an object of $D^b(\coh(\T'))$.
\end{lem}
\begin{proof}
We may assume that $F\simeq \CC_U$ and $G\simeq \CC_V$ for some $U,V \in\T'$.
 Moreover, it is sufficient to consider $U$ and $V$ respectively of the form $U=U_1\times W_1$ and $V=U_2\times W_2$. Then, as a consequence of Proposition 3.4.4 of \cite{KS1} we have $$\rh(\CC_U, \CC_V)\simeq\rh(\CC_{U_1}, \CC_{U_2})\boxtimes\rh(\CC_{W_1},\CC_{W_2}).$$
 Since $\rh(\CC_{U_1}, \CC_{U_2})$ (resp. $\rh(\CC_{W_1},\CC_{W_2}))$ are $\R$-constructible complexes respectively on $X$ and $S$, replacing $\rh(\CC_{U_1}, \CC_{U_2})$

 \noindent (resp. $\rh(\CC_{W_1},\CC_{W_2}))$ by almost free resolutions in the sense of \cite{KS1} we conclude that  $\rh(\CC_U, \CC_V)$ belongs to $D^b(\coh(\T'))$
and the result follows.

\end{proof}

\begin{prop}\label{L:2}
For any $U\in\T'$, $\CC_{U}$ is  $\rho'_*$-acyclic.
\end{prop}
\begin{proof}
It is sufficient to consider $U=U_1\times V_1$, with $U_1\in\op(X_{sa})$ and $V_1\in\op(S_{sa})$. The sheaf $R^j\rho'_*\CC_{U}$ is the sheaf associated to the
presheaf $W \mapsto R^j\Gamma(W;\CC_{U})$ so it is sufficient  to show that
$R^j\Gamma(W;\CC_{U})=0$ for $j \neq 0$ on a family of generators $W$ of the topology of $Y_{\T'}$. In particular, we may assume that $W\in\T'$, so $W=U'\times V'$.

We use the notations of \cite{KS1}.
By the triangulation
theorem there exist a simplicial complex $(K_X,\Delta_X)$, a simplicial complex $(K_S,\Delta_S)$, a
subanalytic homeomorphism $\psi_S:|K_S| \iso S$  compatible with $U_1$  and a subanalytic homeomorphism $\psi_S:|K_X| \iso X$  compatible with $V_1$ such that $U'$  is a finite union of the images by $\psi_X$  of open stars of $|K_X|$ and $V'$ is a a finite union of the images by $\psi_S$  of open stars of $|K_S|$.
So we may assume that $U'$  is the image of an open star compatible with $U_1$ and similarly that $V'$ is the image of an open star compatible with $V_1$.  On the other hand, it is clear by the assumption on $U_1$ (resp. on $V_1$)  and by the construction of an open star with a given center, that $U'\setminus U_1$ always contracts in the center of $U'$ (resp. $V'\setminus V_1$ contracts in the center of $V'$). Indeed, if the center of $U'$ belongs to $U_1$, then $U'\subset U_1$. Otherwise, the contraction of $U'$ in its center restricts to a contraction of $U'\setminus U_1$.
Consider the distinguished triangle
$$
\dt{R\Gamma(W;\CC_{U_1\times V_1})}{R\Gamma(W;\CC_{Y})}{R\Gamma(W;\CC_{Y \setminus U_1\times V_1})}.
$$
It is clear that $U' \times V'$ contracts to the product of the centers respectively of $U'$ and $V'$. On the other hand  the space $(U' \times V') \setminus (U_1 \times V_1)=(U'\setminus U_1)\times V'\cup U'\times (V'\setminus V_1)$
 is a union of closed contractible subspaces such that  the contraction coincide on their intersection, hence it is contractible.
  It follows that  $R\Gamma(W;\CC_{Y}) \simeq R\Gamma(W;\CC_{W})$ and that $R\Gamma(W;\CC_{Y \setminus U_1\times V_1})\simeq R\Gamma(W \setminus U_1\times V_1;\CC_{W \setminus U_1\times V_1})$ are concentrated in degree zero.
This implies that $R\Gamma(W;\CC_{U_1\times V_1})$ is concentrated in degree zero as well.
\end{proof}

In view of Lemma \ref{lem: coh acyclic} we have
\begin{cor}\label{C:T} For any $F \in \coh(\T')$, $F$ is $\rho'_*$-acyclic.
\end{cor}

\begin{nt} Since every $F \in \coh(\T')$ is $\rho'_*$-acyclic and $\rho'_*$ is fully faithful, we can identify
$D^b(\coh(\T'))$ with its image in $D^b(\CC_{Y_{\T'}})$. When there
is no risk of confusion we will write $F$ instead of $\rho'_*F$,
for $F \in D^b(\coh(\T'))$.
\end{nt}

After Corollary \ref{C:T} we have:
\begin{cor} \label{cor:2} Let $F \in D^b(\CC_{Y_{\T'}})$. Then $R\Gamma(W;\imin\eta F) \simeq R\Gamma(W;F)$ for each $ W\in \op(Y_{\T'})$. In particular $\id \iso R\eta_*\imin\eta$.
\end{cor}

\begin{oss}\label{oss:T} Remark that while $\id \iso R\eta_*\imin\eta$, $\imin\eta R\eta_*\iso\id$ does not hold in general. This can be illustrated with the following example: let $X=S=\R$ and let $\overline{B}_1$ be the closed unit ball centered at the origin. It is easy to check that
$$
\eta_*\CC_{\overline{B}_1} \simeq \lind {W \supset \overline{B}_1}\rho'_*\CC_{\overline{W}}
$$
with $W \in \T'$. Then
$$
\imin\eta\eta_*\CC_{\overline{B}_1} \simeq \imin\eta \lind {W \supset \overline{B}_1}\rho'_*\CC_{\overline{W}} \simeq \lind {W \supset \overline{B}_1}\rho_*\CC_{\overline{W}} \not\simeq \CC_{\overline{B}_1},
$$
where the second isomorphism follows from Lemma \ref{lem: structure of inverse image}.
\end{oss}

\begin{lem}\label{L:28}
Let $F \in D^b(\CC_{Y_{\T'}})$ and let $G\in D^b(\coh(\T'))$. Then $$\eta^{-1}\rh(\rho_*'G, F)\simeq \rh(\rho_*G,\eta^{-1}F).$$
\end{lem}
\begin{proof}



 Let $F \in D^b(\CC_{Y_{\T'}})$. By d\'evissage, we may reduce to $F \in \mod(\CC_{Y_{\T'}})$. So $F$ satisfies $F \simeq \lind i\rho'_*F_i$  with $F_i\in\coh({\T'})$ and we can write $\imin\eta F \simeq \lind i\rho_*F_i$.

 We have
\begin{eqnarray*}
H^j\eta^{-1}\rh(\rho_*'G, F) & \simeq & H^j\eta^{-1}\rh(\rho_*'G, \lind i\rho_*'F_i) \\
& \simeq & \lind i\,H^j\eta^{-1}\rho_*'\rh(G, F_i) \\
& \simeq & H^j\lind i\rho_*\rh(G, F_i) \\
& \simeq & H^j\rh(\rho_*G, \lind i\rho_*F_i) \\
& \simeq & H^j\rh(\rho_*G, \eta^{-1}F),
\end{eqnarray*}
where the third isomorphism follows by Lemma \ref{L:26} and Corollary \ref{C:25}.
\end{proof}

We end this section with a result detailing  the behaviour of $\rho_*$ and $\rho'_*$ under tensor product:

\begin{lem}  Let $F \in D^b_{\rc}(\CC_X)$ and $G \in D^b(\CC_S)$. Then
\begin{enumerate}
\item{$\imin {p_1}\rho_*F \otimes \imin {p_2} R\rho_*G \simeq R\rho_*(\imin {p_1}F \otimes \imin {p_2}G)$,}
\item{$\rho'_*\imin {p_1}F \otimes R\rho'_* \imin {p_2} G \simeq R\rho'_*(\imin {p_1}F \otimes \imin {p_2}G)$.}
\end{enumerate}
\end{lem}
\begin{proof} Let us recall that the restriction of $\rho_*$ (resp. $\rho'_*$) to $\R$-constructible sheaves (resp. $\T'$-coherent sheaves) is fully faithful, exact and commutes with $\rh$, $\otimes$ and inverse image. We will often use these facts during the proof of (1) and (2).

(1) We have the chain of isomorphisms
\begin{eqnarray*}
R\rho_*(\imin {p_1}F \otimes \imin {p_2}G) & \simeq & R\rho_*\rh(\imin {p_1}D'F,\imin {p_2}G) \\
& \simeq & \rh(\rho_*\imin {p_1}D'F,R\rho_*\imin {p_2}G) \\
& \simeq & \rh(\imin {p_1}D'\rho_*F,\imin {p_2}R\rho_*G) \\
& \simeq & \imin {p_1}\rho_*F \otimes \imin {p_2} R\rho_*G.
\end{eqnarray*}
The first isomorphism follows from Proposition 3.4.4 of \cite{KS1}, the second one from Proposition 2.2.1 of \cite{P}, the third isomorphism follows from the fact that $\imin {p_2}(\cdot) \otimes p_2^!\CC_{X \times S} \simeq p_2^!(\cdot)$ (Proposition 2.4.9 of \cite{P}) and that $p_2^!$ commutes with $R\rho_*$ (Proposition 2.4.5 of \cite{P}) and the fourth one follows from Lemma 5.3.9 of \cite{PM}.

(2) We prove the assertion in several steps. Recall that $\rho'=\eta \circ \rho$, where $\eta:X_{sa} \to X_{\T'}$ is the natural functor of sites.
\begin{itemize}
\item[(2a)] Let $F \in \coh(\T')$ and $G \in \mod(\CC_{S_{sa}})$.  Then $G=\lind i \rho_*G_i$ with $G_i \in \mod_{\rc}(\CC_S)$. We have
$$
\lind i \imin {p_2}\rho_*G \simeq \lind i \rho_*\imin {p_2}G_i \simeq \lind i \imin\eta\rho'_*\imin {p_2}G_i \simeq \imin\eta \lind i \rho'_*\imin {p_2}G_i.
$$
The first isomorphism follows from Proposition 1.3.3 of \cite{P}, the second one from Corollary \ref{C:25}, the third one since inverse images commute with $\Lind$. Therefore $\imin {p_2}G \simeq \imin\eta G'$ with $G' \simeq \lind i \rho'_*\imin {p_2}G_i \in \mod(\CC_{Y_{\T'}})$. We have
$$
R\eta_*(\rho_*F \otimes \imin\eta G') \simeq R\eta_*\imin\eta(\rho'_*F \otimes G') \simeq \rho'_*F \otimes G'
$$
The first isomorphism follows since inverse images commute with $\otimes$ and $\rho_* \simeq \imin\eta \circ \rho'_*$ on $\coh(\T')$. The second isomorphism follows from Proposition \ref{L:2}. Hence $R\eta_*(\rho_*F \otimes \imin {p_2}G)$ is concentrated in degree 0.

\item[(2b)] Let $F \in D^b(\coh(\T'))$ and $G \in D^b(\CC_{S_{sa}})$. We shall prove that $$\eta_*(\rho_*F \otimes \imin {p_2}G) \simeq \eta_*\rho_*F \otimes \eta_*\imin {p_2}G$$ (here we use the last asssertion in (2a) to replaced $R\eta_*$ by $\eta_*$). By d\'evissage we may reduce to $F,G$ concentrated in degree zero. By (2a), we have $\imin {p_2}G\simeq\imin\eta G'\simeq\imin\eta\eta_*\imin\eta G'\simeq\imin\eta\eta_*\imin {p_2}G$ with $G' \in \mod(\CC_{Y_{\T'}})$ (the second isomorphism follows from Corollary \ref{cor:2}).
   In view of the preceding arguments, we have the chain of isomorphisms
    \begin{eqnarray*}
    \eta_*(\rho_*F \otimes \imin {p_2}G) & \simeq & \eta_*(\imin\eta\rho'_*F \otimes \imin\eta\eta_*\imin {p_2} G) \\
    & \simeq & \eta_*\imin\eta(\rho'_*F \otimes \eta_*\imin {p_2} G) \\
    & \simeq & \rho'_*F \otimes \eta_*\imin {p_2} G \\
    & \simeq & \eta_*\rho_*F \otimes \eta_*\imin{p_2}G.
    \end{eqnarray*}
\item[(2c)] Let $F \in D^b_{\rc}(\CC_X)$ and $G \in D^b(\CC_S)$. Then
    \begin{eqnarray*}
    R\eta_*R\rho_*(\imin {p_1}F \otimes \imin {p_2}G) & \simeq & R\eta_*(\imin {p_1}\rho_*F \otimes \imin {p_2}R\rho_*G) \\
    & \simeq & \eta_*\rho_*\imin {p_1}F \otimes R\eta_*\imin {p_2}R\rho_*G \\
    & \simeq & \eta_*\rho_*\imin {p_1}F \otimes R\eta_*R\rho_*\imin {p_2}G.
    \end{eqnarray*}
    The first isomorphism follows from (1), the second one from (2b) and the third one from the fact that $\imin {p_2}$ commutes with $R\rho_*$ (see the proof of (1)).
\end{itemize}
\end{proof}

\section{Construction of relative subanalytic sheaves}\label{S:2}


Let $X$ and $S$ be two real analytic manifolds. Let  be given a subanalytic sheaf $F$ on $(X\times S)_{sa}$. 
We shall denote by  $F^{S,\sharp}$  the sheaf on $X_{sa} \times S_{sa}$ associated to the presheaf
\begin{eqnarray*}
\op(X_{sa} \times S_{sa}) & \to & \mod(\CC) \\
U \times V & \mapsto & \Gamma(X \times V;\imin\rho\Gamma_{U \times S}F) \\
& \simeq & \Ho(\CC_U \boxtimes \rho_!\CC_V,F) \\
& \simeq & \lpro {W \subset\subset V, W\in\op^c(S_{sa})}\Gamma(U \times W;F).
\end{eqnarray*}

We set \begin{equation}\label{E:10}
F^{S}:=\imin\eta F^{S,\sharp}
\end{equation}
 and call it the relative sheaf associated to $F$. It is a sheaf on $(X \times S)_{sa}$.
 It is easy to check that $(\cdot)^S$ defines a left exact functor on $\mod(\CC_{(X \times S)_{sa}})$.

 We will denote by $(\cdot)^{RS,\sharp}$ and $(\cdot)^{RS} \simeq \imin\eta \circ (\cdot)^{RS,\sharp}$ the associated right derived functors.

According to Lemma \ref{L:28} we get:
\begin{prop}\label{P:26}
  For each $G\in D^b_{\rc}(\CC_X)$, $H\in D^b_{\rc}(\CC_S)$ and $F\in D^b(\CC_{(X \times S)_{sa}})$, we have $$\eta^{-1}\rh(\rho_*'(G\boxtimes H), F^{RS,\sharp})\simeq \rh(\rho_*(G\boxtimes H), F^{RS}).$$
\end{prop}

 The following Lemmas are steps to prove Proposition \ref{lem 9} below:

\begin{lem} \label{lem 5}Let $U \in \op(X_{sa})$, $V \in \op(S_{sa})$. Then
\begin{eqnarray*}
\Gamma(U \times V; F^{S}) & \simeq & \Gamma(X \times V;\imin\rho\Gamma_{U \times S}F) \\
& \simeq & \Ho(\CC_U \boxtimes \rho_!\CC_V,F).
\end{eqnarray*}
\end{lem}
\begin{proof} The second isomorphism follows by adjunction. Let us prove the first one. By $(2)$ of Lemma \ref{lem: coh acyclic} it is enough to check that $\Gamma(U \times V;F^{S,\sharp})=\Gamma(X \times V;\imin\rho\Gamma_{U \times S}F).$

1) We first suppose that $U \times V$ is relatively compact. Let $s \in \Gamma(U \times V; F^{S,\sharp})$. Then $s$ is defined by a finite family
$s_i \in
\lpro {W_i \subset\subset V_i} \Gamma(U_i \times W_i;F)$, $i \in I$ where $\{U_i\}$ (resp. $\{V_i\}$) is a covering of $U$ (resp. $V$) in $X_{sa}$ (resp. $S_{sa}$), such that $s_i = s_j$ on $(U_i \times V_i) \cap (U_j \times V_j)$.

By Lemma 3.6 of \cite{vdD}, there exists a
refinement $\{V'_i\}$ of $\{V_i\}$ in $S_{sa}$ such that
$\overline{V'_i} \cap V \subset V_i$. Now we have the following obvious facts:
\begin{itemize}
\item[(i)] If, for a given $W'_i \in \op(S_{sa})$,  $W'_i \subset\subset V'_i$, then $W'_i \subset\subset V$,
\item[(ii)] If, for a given $W\in\op(S_{sa})$, $W \subset\subset V $, then $V'_i \cap W \subset\subset V_i$.
\end{itemize}
This implies that the restriction $\lpro {W_i \subset\subset V_i} \Gamma(U_i \times W_i;F) \to \lpro {W'_i \subset\subset V'_i} \Gamma(U_i \times W'_i;F)$ factors through $\lpro {W \subset\subset V}\Gamma(U_i \times (W \cap V'_i);F)$.
Therefore $s|_{{U_i}\times V'_i}$  extends to a section  of
$\Gamma(X\times V;\imin\rho\Gamma_{U_i \times V'_i}F) \simeq \lpro {W \subset\subset V}\Gamma(U_i \times (W \cap V'_i);F)$. Set $U_{ij}=U_i \cap U_j$ and $V'_{ij}=V'_i \cap V'_j$.
The exact sequence
$$
\oplus_{i \neq j \in I}\CC_{U_{ij} \times V'_{ij}} \to \oplus_{k \in I}\CC_{U_k \times V'_k} \to \CC_{U \times V} \to 0
$$
defines an exact sequence
$$
\lexs{\Gamma(X \times V;\imin \rho\Gamma_{U \times V}F)}{\oplus_k\Gamma(X \times V;\imin\rho\Gamma_{U_k \times V'_k}F)}{\oplus_{i \neq j}\Gamma(X \times V;\imin\rho\Gamma_{U_{ij} \times V'_{ij}F}}).
$$
Then the $s_i$'s glue to a section of
$\Gamma(X \times V;\imin \rho\Gamma_{U \times V}F) \simeq \Gamma(X \times V;\imin\rho\Gamma_{U \times S}F)$ as required.


2) Suppose that $U \in \op(X_{sa})$ and $V \in \op^c(S_{sa})$. Then $U=\bigcup_{n \in \N}(U \cap U_n)$ where $\{U_n\}_{n \in \N}$ belongs to $\cov(X_{sa})$ and satisfies $U_n\subset\subset U_{n+1}$. Then
\begin{eqnarray*}
\Gamma(U \times V;F^{S,\sharp}) & \simeq & \lpro n \Gamma(U_n \times V;F^{S,\sharp})
\\
& \simeq & \lpro n \Gamma(X \times V;\imin\rho\Gamma_{U_n \times S}F) \\
&  \simeq & \Gamma(X \times V;\imin\rho\lpro n \Gamma_{U_n \times S}F)\\
&  \simeq & \Gamma(X \times V;\imin\rho\Gamma_{U \times S}F).
\end{eqnarray*}

3) Now consider the general case. Let $s \in \Gamma(U \times V; F^{S,\sharp})$. It is defined by a countable family $s_n \in \Gamma(U \times V_n; F^{S,\sharp})=\Gamma(X \times V_n;\imin\rho\Gamma_{U \times S}F)$ where $\{V_n\}_{n \in \N}$ is a covering of $V$ in $S_{sa}$ such that $\overline{V}_n \cap V \subset V_{n+1}$. Then there exists a refinement $\{V'_n\}$ of $\{V_n\}$ in $S_{sa}$ with $V_{n-1} \subset \overline{V'}_n \cap V \subset V_n$. Arguing as in 1), the restriction $s_n|_{U \times V'_n}$ belongs to $\Gamma(X \times V; \imin\rho\Gamma_{U \times V'_n}F)$ and $s_n=s_{n+1}$ on $U \times V'_n$. Hence they glue to $s \in \Gamma(X \times V;\imin\rho\Gamma_{U \times V}F) \simeq \Gamma(X \times V;\imin\rho\Gamma_{U \times S}F)$ as required.
\end{proof}

With Proposition 6.5.1 of \cite{KS3} (applied on $X$ and $S$ separately) as a tool we now prove the following result:

\begin{lem} \label{lem 6} Let $G \in \mod_{\rc}(\CC_X)$, $H \in \mod_{\rc}(\CC_S)$. Let $F \in \mod(\CC_{(X \times S)_{sa}})$.
Then
$$
\Ho(G \boxtimes H,F^S) \simeq \Ho(G \boxtimes \rho_!H,F)\simeq \Ho(\CC_X \boxtimes H,\imin\rho\ho(G \boxtimes \CC_S,F)) .
$$
\end{lem}

\begin{proof} The right hand isomorphism follows by adjunction. Let us prove the left hand isomorphism.

1) Suppose at first that $G$ and $H$ have compact support. By Proposition 6.5.1 of \cite{KS3} the functor $U \mapsto \Ho(\CC_U \boxtimes \rho_!\CC_V,F)\simeq\Ho(\CC_U \boxtimes \CC_V,F^S)$ extends uniquely to a functor $\mod^c_{\rc}(\CC_X) \to \mod(\CC)$. This implies $$\Ho(G \boxtimes \rho_!\CC_V,F) \simeq \Ho(G \boxtimes \CC_V,F^S).$$
Similarly, the functor $V \mapsto \Ho(G \boxtimes \rho_!\CC_V,F)\simeq\Ho(G \boxtimes \CC_V,F^S)$  extends uniquely to a functor $\mod^c_{\rc}(\CC_S) \to \mod(\CC)$. This implies $$\Ho(G \boxtimes \rho_!H,F) \simeq \Ho(G \boxtimes H,F^S).$$

2) Let us consider the general case. Let $\{U_n\}_{n \in \N}$ (resp. $\{V_n\}_{n \in \N}$ be a covering of $X_{sa}$ (resp. $S_{sa}$) such that $U_n \subset\subset U_{n+1}$ (resp. $V_n \subset\subset V_{n+1}$) for each $n$. We have
\begin{eqnarray*}
\Ho(G \boxtimes H,F^S) & \simeq & \lpro n\Ho(G_{U_n} \boxtimes H_{V_n},F^S) \\
& \simeq & \lpro n\Ho(G_{U_n} \boxtimes \rho_!(H_{V_n}),F) \\
& \simeq & \lpro n\Ho(G_{U_n} \boxtimes (\rho_!H)_{V_n},F) \\
& \simeq & \lpro n\Gamma(U_n \times V_n;\ho(G \boxtimes \rho_!H,F))
\\
& \simeq & \Gamma(X;\ho(G \boxtimes \rho_!H,F)) \\
& \simeq & \Ho(G \boxtimes \rho_!H,F).
\end{eqnarray*}
The second isomorphism follows from 1).To prove the third one we remark that the morphism $(\rho_!H)_{V_n} \to (\rho_!H)_{V_{n+1}}$ factors through $\rho_!(H_{V_{n+1}}) \simeq \lind {W \subset\subset V_{n+1}} (\rho_!H)_W$. The desired isomorphism then follows passing to the limit on $n \in \N$.
\end{proof}


We shall now prepare the steps to the main result of this note, Proposition \ref{lem 9} below. Recall (cf \cite{EP}) that $F \in \mod(\CC_{(X_{sa} \times S_{sa})})$ is $\T'$-flabby if the restriction morphism $\Gamma(X \times S;F) \to \Gamma(W;F)$ is surjective for each $W \in \T'$. $\T'$-flabby objects are $\Ho(G,\cdot)$-acyclic for each $G \in \coh(\T')$.

\begin{lem} \label{lem 6 1/2} Let $F \in \mod(\CC_{(X \times S)_{sa}})$ be injective. Then $F^{S,\sharp}$ is $\T'$-flabby.
\end{lem}
\begin{proof}

Let us consider $W=\bigcup_{i=1}^n (U_i \times V_i)$, with $U_i \in \op^c(X_{sa})$ and $V_i \in \op^c(S_{sa})$. For $i \in \{1,\dots,n\}$, set
\begin{eqnarray*}
K_{i} & = & \lind {W_1 \subset\subset V_1} \cdots \lind {W_i \subset\subset V_i}\rho_*\CC_{(U_1 \times W_1) \cup \cdots \cup (U_i \times W_i)}.
\end{eqnarray*}

1) We first prove that
$$
\Gamma\left(W;F^S\right) \simeq \Ho(K_n,F).$$

We argue by induction on $n$.

For $n=1$ the result follows from Lemma \ref{lem 5}.

$n-1 \Rightarrow n$: Set $K'_{n-1} = K_{n-1} \otimes (\CC_{U_n} \boxtimes \rho_!\CC_{V_n})$. We have
\begin{eqnarray*}
K'_{n-1} & \simeq & \lind {\substack{W_1 \subset\subset V_1 \\ W_n \subset\subset V_n}} \cdots \lind {\substack{W_{n-1} \subset\subset V_{n-1} \\ W_n \subset\subset V_n}}\rho_*\CC_{((U_1 \cap U_n) \times (W_1 \cap W_n)) \cup \cdots \cup ((U_{n-1} \cap U_n) \times (W_{n-1} \cap W_n))} \\
& \simeq & \lind {W'_1 \subset\subset V_1 \cap V_n} \cdots \lind {W'_{n-1} \subset\subset V_{n-1} \cap V_n}\rho_*\CC_{((U_1 \cap U_n) \times W'_1) \cup \cdots \cup ((U_{n-1} \cap U_n) \times W'_{n-1})}.
\end{eqnarray*}
We have an exact sequence
$$
\exs{K'_{n-1}}{K_{n-1} \oplus (\CC_{U_n} \boxtimes \rho_!\CC_{V_n})}{K_n}.
$$
Applying the functor $\Ho(\cdot,F)$ and using the induction hypothesis on $K'_{n-1}$ and $K_{n-1}$ we obtain
\begin{eqnarray*}
\Gamma\left(\bigcup_{i=1}^{n-1} ((U_i \cap U_n) \times (V_i \cap V_n));F^S\right) & \simeq & \Ho(K'_{n-1},F) \\
\Gamma\left(\bigcup_{i=1}^{n-1} (U_i \times V_i);F^S\right) & \simeq & \Ho(K_{n-1},F).
\end{eqnarray*}
Hence $\Gamma\left(\bigcup_{i=1}^n (U_i \times V_i);F^S\right) \simeq \Ho(K_n,F)$, as required.

2) Consider the monomorphism $0 \to K_n \to \CC_X \boxtimes \CC_S$. Since $F$ is injective we obtain a surjection
$$
\Ho(\CC_X \boxtimes \CC_S,F) \to \Ho(K_n,F) \to 0
$$
and the result follows.
\end{proof}

\begin{cor} \label{cor 6 1/2}  Let $G \in D^b_{\rc}(\CC_X)$, $H \in D^b_{\rc}(\CC_S)$. Let $F \in \mod(\CC_{(X \times S)_{sa}})$ be injective. Then $F^S$ is $\Ho(G \boxtimes H, \cdot)$-acyclic.
\end{cor}
\begin{proof}
First remark that, $F$ being injective, we have $F^{RS,\sharp} \simeq F^{S,\sharp}$ and $F^{RS} \simeq F^S$.
By Proposition \ref{P:26} and Proposition \ref{L:2} we have
\begin{eqnarray*}
\Rh(G \boxtimes H,F^S) & \simeq & R\Gamma(X \times S;\rh(G \boxtimes H,F^S)) \\
& \simeq & R\Gamma(X \times S;\imin\eta\rh(G \boxtimes H,F^{S,\sharp})) \\
& \simeq & R\Gamma(X \times S;\rh(G \boxtimes H,F^{S,\sharp})) \\
& \simeq & \Rh(G \boxtimes H,F^{S,\sharp}).
\end{eqnarray*}
Lemma \ref{lem 6 1/2} implies that $F^{S,\sharp}$ is $\Ho(G \boxtimes H, \cdot)$-acyclic and the result follows.
\end{proof}

\begin{lem} \label{lem 7} Let $G \in D^b_{\rc}(\CC_X)$, $H \in D^b_{\rc}(\CC_S)$. Let $F \in D^b(\CC_{(X \times S)_{sa}})$. Then
$$
\Rh(G \boxtimes H,F^{RS}) \simeq \Rh(G \boxtimes \rho_!H,F)\simeq \Rh(\CC_X \boxtimes H,\imin\rho\rh(G \boxtimes \CC_S,F)) .
$$
\end{lem}
\begin{proof} The right hand isomorphism follows by adjunction.
 Let us prove the left hand isomorphism.




By Corollary \ref{cor 6 1/2} we see that that $(\cdot)^S$ sends injective objects of $\mod(\CC_{(X \times S)_{sa}})$ to $\Ho(G \boxtimes H,\cdot)$-acyclic objects, for $G \in \mod_{\rc}(\CC_X)$, $H \in \mod_{\rc}(\CC_S)$.

Therefore, we may reduce to $F$ injective and $G,H$ concentrated in degree 0. Then the result follows from Lemma \ref{lem 6}. \end{proof}

Remark that if $K \in \mod(\CC_{(X \times S)_{sa}})$ then $\imin\rho K \osi \imin\rho \imin\eta \eta_*K$. Indeed, for each $y \in X \times S$,
$$
(\imin\rho K)_y\simeq \lind {U \times V \ni y}K(U \times V) \simeq (\imin\rho \imin\eta \eta_*K)_y
$$
with $U \in \op(X_{sa})$, $V \in \op(S_{sa})$.

\begin{prop} \label{lem 9} Let $F \in D^b(\CC_{(X \times S)_{sa}})$. Let $G \in D^b_{\rc}(\CC_X)$ and $H \in D^b_{\rc}(\CC_S)$. Then
\begin{eqnarray*}
\imin \rho \rh(G \boxtimes H,F^{RS}) & \simeq  & \imin \rho \rh(G \boxtimes \rho_!H,F) \\
& \simeq & \rh(\CC_X \boxtimes H,\imin\rho\rh(G \boxtimes \CC_S,F)).
\end{eqnarray*}
In particular, when $G=\CC_X$ and $H=\CC_S$ we have $\imin\rho F \simeq \imin\rho F^{RS}\simeq \rho'^{-1} F^{RS,\sharp}$.
\end{prop}
\begin{proof} The second isomorphism follows by adjunction. Let us prove the first one.

1) Let us first suppose that $F,G,H$ are concentrated in degree zero. Hence,  by the remark above, to  any morphism $$ \eta_* \ho(G \boxtimes \rho_!H,F) \to \eta_* \ho(G \boxtimes H,F^S)$$ one associates a morphism $$\imin\rho \ho(G \boxtimes \rho_!H,F) \to \imin\rho \ho(G \boxtimes H,F^S).$$

Note that the natural morphism of functors $\rho_!(H_V) \to (\rho_!H)_V$ induces a morphism $\Ho(G_U \boxtimes (\rho_!H)_V,F) \to \Ho(G_U \boxtimes \rho_!(H_V),F)$  hence a morphism $\psi: \eta_* \ho(G \boxtimes \rho_!H,F) \to \eta_* \ho(G \boxtimes H,F^S)$, which defines a morphism  $\imin\rho \ho(G \boxtimes \rho_!H,F) \to \imin\rho \ho(G \boxtimes H,F^S)$.

 Let us check on the  fibers that it is an isomophism.
 Let $y \in X \times S$, then
\begin{eqnarray*}
(\imin\rho \ho(G \boxtimes \rho_!H,F))_y & \simeq & \lind {U \times V \ni y}\Ho(G_U \boxtimes (\rho_!H)_V,F) \\
& \simeq & \lind {U \times V \ni y}\lpro {W \subset\subset V}\Ho(G_U \boxtimes (\rho_!H)_W,F) \\
& \simeq & \lind {U \times V \ni y}\Ho(G_U \boxtimes \rho_!(H_V),F) \\
& \simeq & (\imin\rho \ho(G \boxtimes H,F^S))_y
\end{eqnarray*}
with $U \in \op(X_{sa})$, $V,W \in \op(S_{sa})$.

2) Suppose now that $F$ is injective and that $G,H$ are concentrated in degree 0. Let $U \in \op(X_{sa})$, $V \in \op(S_{sa})$. The complex
\begin{eqnarray*}
R\Gamma(U \times V;\rh(G \boxtimes H,F^S)) & \simeq & \Rh(G_U \boxtimes H_V,F^S) \end{eqnarray*}
is concentrated in degree 0 by Corollary \ref{cor 6 1/2}. Then $F^S$ is $\rh(G \boxtimes H,\cdot)$-acyclic.

3) Let $G \in D^b_{\rc}(\CC_X)$ and $H \in D^b_{\rc}(\CC_S)$. Let $F \in D^b(\CC_{(X \times S)_{sa}})$ and let $I^\bullet$ be a complex of injective objects quasi-isomorphic to $F$. Then
\begin{eqnarray*}
\imin\rho \rh(G \boxtimes \rho_!H,F) & \simeq & \imin\rho \ho(G \boxtimes \rho_!H,I^\bullet) \\
& \simeq & \imin \rho \ho(G \boxtimes H,(I^\bullet)^S) \\
& \simeq & \imin \rho \rh(G \boxtimes H,F^{RS}),
\end{eqnarray*}
where the second isomorphism follows from 1) and the third one from 2). \end{proof}


We end this section with the following result on the acyclicity for the functor $(\cdot)^S$ which will be needed in the sequel:
\begin{prop} \label{lem 8} Suppose that $F \in \mod(\rho_!\C^\infty_{X \times S})$ is $\Gamma(W;\cdot)$-acyclic for each $W \in \op((X \times S)_{sa})$. Then for each $U \in \op(X_{sa})$, $V \in \op(S_{sa})$ we have $R^k\Gamma(U \times V;F^{RS,\sharp}) = R^k\Gamma(U \times V;F^{RS}) = 0$ if $k \neq 0$.
\end{prop}
\begin{proof} By Lemma \ref{lem 7}  we have $R \Gamma(U \times V; F^{RS,\sharp}) \simeq R\Gamma(U \times V;F^{RS}) \simeq R\Gamma(X \times V; \imin\rho R\Gamma_{U \times S}F)$. Since $F$ is $\Gamma(W;\cdot)$-acyclic for each $W \in \op((X \times S)_{sa})$, the complex $R\Gamma_{U \times S}F$ is concentrated in degree zero. Since $F$ is a $\rho_!\C^\infty_{X \times S}$-module,  $\imin\rho\Gamma_{U \times S}F$ is a $\C^\infty_{X \times S}$-module, hence c-soft and $\Gamma(X \times V;\cdot)$-acyclic. This shows the result. \end{proof}

\begin{cor} \label{cor 8} Suppose that $F \in \mod(\rho_!\C^\infty_{X \times S})$ is $\Gamma(W;\cdot)$-acyclic for each $W \in \op((X \times S)_{sa})$. Then 
$F$ is $(\cdot)^{S,\sharp}$-acyclic and $(\cdot)^S$-acyclic.
\end{cor}
\begin{proof} Being $(\cdot)^{RS} \simeq \imin\eta \circ (\cdot)^{RS,\sharp}$, it is enough to show that $H^kF^{RS,\sharp}=0$ if $k \neq 0$. It is enough to prove that $R^k\Gamma(W;F^{RS,\sharp})=0$ if $k \neq 0$ on a basis for the topology of $(X \times S)_{\T'}$. Since the products $U \times V$ with $U \in \op(X_{sa})$, $V \in \op(S_{sa})$ form a basis, the result follows from Proposition \ref{lem 8}.
\end{proof}

\section{The  sheaves  $\C_{X\times S}^{\infty, t, S}$, $\db_{X\times S}^{t,S}$, $\C_{X\times S}^{\infty,\mathrm{w},S}$, $\OO_{X\times S}^{t,S}$ and $\OO_{X\times S}^{\mathrm{w},S}$}\label{S:3}
Let $X$  and $S$ be real analytic manifolds.
The construction given by (\ref{E:10}) allows us to introduce the following sheaves:
\begin{enumerate}
\item{$\C^{\infty,t,S}_{X\times S}:=(\C^{\infty,t}_{X\times S})^S$ as the relative sheaf associated to $\C^{\infty,t}_{X\times S}$,}
\item{$\db_{X\times S}^{t,S}:=(\dbt_{X \times S})^S$ as the relative sheaf associated  to $\db^{t}_{X\times S}$,}
\item{$\C_{X\times S}^{\infty,\mathrm{w},S}:=(\CW_{X \times S})^S$ as the relative sheaf associated to $\C_{X\times S}^{\infty,\mathrm{w}}$.}

\end{enumerate}

We then derive from Lemma \ref{lem 5}:
\begin{prop} \label{prop 6}Let $U \in \op(X_{sa})$, $V \in \op(S_{sa})$. Then
\begin{enumerate}
\item{$\Gamma(U \times V;\C^{\infty, t,S}_{X\times S})\simeq\Gamma(X \times V;\imin\rho\Gamma_{U \times S}\C^{\infty,t}_{X \times S}) \\ \ \ \ \ \simeq \Gamma(X \times V;\tho(\CC_{U \times S},\C^\infty_{X \times S})),$} \\
\item{$\Gamma(U \times V;\db^{t,S}_{X \times S})\simeq\Gamma(X \times V;\imin\rho\Gamma_{U \times S}\dbt_{X \times S}) \\ \ \ \ \ \simeq \Gamma(X \times V;\tho(\CC_{U \times S},\db_{X \times S})),$} \\
\item{$\Gamma(U \times V;\C^{\infty,\mathrm{w},S}_{X \times S})\simeq\Gamma(X \times V;\imin\rho\Gamma_{U \times S}\C^{\infty,\mathrm{w}}_{X \times S}) \\ \ \ \ \ \simeq \Gamma(X \times V;H^0D'\CC_U \boxtimes \CC_S \wtens \C^\infty_{X \times S}).$}
\end{enumerate}
\end{prop}

We can now state:

\begin{prop} \label{prop 8}

i) Suppose that $\mathcal{F} =\db^t_{X \times S},\C^{\infty,t}_{X \times S}, \C^{\infty,\mathrm{w}}_{X \times S}$. Then $\mathcal{F}$ is $(\cdot)^{\sharp,S}$-acyclic and hence $(\cdot)^S$-acyclic. Moreover $\db^t_{X \times S},\C^{\infty,t}_{X \times S}$ are $\Gamma(U \times V;\cdot)$-acyclic for each $U \in \op(X_{sa})$, $V \in \op(S_{sa})$.

ii)  $\C^{\infty,\mathrm{w},S}_{X \times S}$ is $\Gamma(U \times V;\cdot)$-acyclic for each $U \in \op(X_{sa})$ locally cohomologically trivial and $V \in \op(S_{sa})$.
\end{prop}

Indeed i) is a consequence of  Proposition \ref{lem 8} and ii) follows from Lemma \ref{lem 7} and  Proposition \ref{lem 8}.

Applying   Proposition \ref{lem 9} and Proposition 7.2.6 of \cite{KS3} we conclude:
\begin{prop} \label{prop 7}Let $G \in D^b_{\rc}(\CC_X)$, $H \in D^b_{\rc}(\CC_S)$. Then
\begin{enumerate}
\item{$\imin\rho\rh(G \boxtimes H,\C^{\infty,t,S}_{X \times S}) \simeq \imin\rho\rh(G \boxtimes \rho_!H,\C^{\infty,t}_{X \times S}) \\ \ \ \ \ \simeq \rh(\CC_X \boxtimes H,\tho(G \boxtimes \CC_S,\C^\infty_{X \times S})),$} \\
\item{$\imin\rho\rh(G \boxtimes H,\db^{t,S}_{X \times S}) \simeq \imin\rho\rh(G \boxtimes \rho_!H,\dbt_{X \times S}) \\ \ \ \ \ \simeq \rh(\CC_X \boxtimes H,\tho(G \boxtimes \CC_S,\db_{X \times S})),$} \\
\item{$\imin\rho\rh(G \boxtimes H,\C^{\infty,\mathrm{w},S}_{X \times S}) \simeq \imin\rho\rh(G \boxtimes \rho_!H,\CW_{X \times S}) \\ \ \ \ \ \simeq \rh(\CC_X \boxtimes H,D'G \boxtimes \CC_S \wtens \C^\infty_{X \times S}).$}
\end{enumerate}
In particular, when $G=\CC_X$ and $H=\CC_S$ we have $\imin\rho\C^{\infty,t,S}_{X \times S} \simeq \C^\infty_{X \times S}$, $\imin\rho\db^{t,S}_{X \times S} \simeq \db_{X \times S}$, $\imin\rho\C^{\infty,\mathrm{w},S}_{X \times S} \simeq \C^\infty_{X \times S}$.
\end{prop}

\begin{lem}\label{L:3}
There is a natural action of $\rho_!\D_{X\times S}$ on $\db^{t,S}_{X\times S}$, on $\C^{\infty,t,S}_{X\times S}$ and on $\C_{X\times S}^{\infty,\mathrm{w},S}$.
\end{lem}

\begin{proof}
The proof being similar in the three cases, we just give it for the first one.
By Proposition 3.2.1 of \cite{P}, it is enough to check that the presheaf $$
\eta^\dagger \db^{t,S,\sharp}_{X\times S}(W)= \lind {W \subset W'} \db^{t,S,\sharp}_{X\times S}(W')
$$
with $W' \in \op(X_{sa} \times S_{sa})$, is a presheaf over the presheaf of rings $W\mapsto \Gamma(\overline {W};\D_{X\times S})$. Setting $W'=U\times V$,  we have by Lemma 2.1 $$\Gamma(U \times V;\db^{t,S,\sharp}_{X\times S})=\Gamma(U \times V;\dbt_{X \times V}).$$

We may assume that  $W\in\op^c((X\times S)_{sa})$. Thus we can cover $\overline{W}$ by finite many open subsets $\{U_{i}\times V_i\}$, $\{U'_i\times V'_i\}$ with $U_i\times V_i, U'_i\times V'_i\in \op^c(X_{sa}\times S_{sa})$ sufficiently small and such that $U'_i\times V'_i\subset\subset U_i\times V_i$. Given $P\in\Gamma(\overline{W};\D_{X\times S})$, for a convenient covering  $\{U_{i}\times V_i\}$  as above, $P$ is defined on $\bigcup_iU_i\times V_i$. We then deduce the action of $P$ on $\lind {W \subset W'} \Gamma(W'; \db^{t,S,\sharp}_{X\times S})$   as the image of the gluing of the actions  on each $\Gamma(U'_i\times V'_i; \db^t_{X \times V'_i})$.
\end{proof}

Let us now assume that $X$ and $S$  are complex manifolds and consider the projection $f:X\times S\to S$. Let us denote as usual by $\overline{X}\times\overline{S}$ the complex conjugate manifold. Identifying the underlying real analytic manifold $X_{\R} \times S_{\R}$ to the diagonal of $(X \times S) \times (\overline{X} \times \overline{S})$, we have:

\begin{lem}\label{O_S}
 $\rho_*\imin f\OO_S$ (resp. $\rho'_*f^{-1}\OO_S$) acts on $\db^{t,S}_{X \times S}$, on $\C^{\infty,t}_{X \times S}$ and on $\C_{X \times S}^{\mathrm{w},S}$ (resp. on $\db^{t,S,\sharp}_{X \times S}$, on $\C^{\infty,t,S.\sharp}_{X \times S}$ and on $\C_{X \times S}^{\mathrm{w},S,\sharp}$).
\end{lem}
\begin{proof}
 To prove the action of $\rho_*\imin f\OO_S$ it is sufficient to check that $\rho_*\imin f\OO_S(W)$ acts on $\db^{t,S}_{X \times S}(W)$ on a basis for the topology of $(X \times S)_{sa}$. Since every relatively compact subanalytic open subset of $X \times S$ can be covered by open cells (cf \cite{W}), we may suppose that $W$ is an open cell such that $f|_W:W \to f(W)$ is the restriction of a composition of projections $f_j:\R^j \times f(W) \to \R^{j-1} \times f(W)$ and the fibers of $f$ intersected with $W$ are contractible or empty. In this case we have $\rho_*\imin f\OO_S(W)=\OO_S(f(W))$ and $\OO_S(f(W))$ acts on $\db^{t,S}_{X \times S}(W)$, since $\db^{t,S}_{X \times S}(W)$ has no growth conditions on the boundary of $\imin f(f(W))$.
 The proof is similar for $\C^{\infty,t}_{X \times S}$ and for  $\C_{X \times S}^{\mathrm{w},S}$.

 Similarly, to prove the action of $\rho'_*f^{-1}\OO_S$ it is sufficient to check that $\rho'_*\imin f\OO_S(U\times V)\simeq \OO_S(V)$ acts on $\db^{t,S}_{X \times S}(U\times V)$ where $U\in\op(Y_{sa})$ is assumed to be contractible and $V\in\op(S_{sa}).$ Since $\db^{t,S}_{X \times S}(U\times V)\simeq \lpro {W \subset\subset V, W\in\op^c(S_{sa})}\Gamma(U \times W; \db_{X \times S}^t)$ the statement is clear. \end{proof}

The construction given by (\ref{E:10}) allows us to introduce the following objects of $D^b(\CC_{(X \times S)_{sa}})$:
\begin{enumerate}
\item{$\OO_{X \times S}^{t,S}:=(\ot_{X \times S})^{RS}$, the relative sheaf associated  to $\ot_{X \times S}$, that is $$\OO_{X \times S}^{t,S}  \simeq  (\rh_{\rho_!\D_{\overline{X} \times \overline{S}}}(\rho_!\OO_{\overline{X} \times \overline{S}},\dbt_{X \times S}))^{RS} \simeq  (\rh_{\rho_!\D_{\overline{X} \times \overline{S}}}(\rho_!\OO_{\overline{X} \times \overline{S}},\C^{\infty,t}_{X \times S}))^{RS}$$}
\item{$\OO_{X \times S}^{\mathrm{w},S}:=(\OW_{X \times S})^{RS}$, the relative sheaf associated to $\OW_{X \times S}$, that is $$\OO_{X \times S}^{\mathrm{w},S} \simeq  (\rh_{\rho_!\D_{\overline{X} \times \overline{S}}}(\rho_!\OO_{\overline{X} \times \overline{S}},\CW_{X \times S}))^{RS}.$$}

\end{enumerate}

The  exactness of $\rho_!$   together with Proposition
\ref{prop 8}
 allow to conclude:
\begin{prop} We have the following isomorphisms in $D^b(\CC_{(X \times S)_{sa}})$.
\begin{eqnarray*}
\OO_{X \times S}^{t,S} \simeq \rh_{\rho_!\D_{\overline{X} \times \overline{S}}}(\rho_!\OO_{\overline{X} \times \overline{S}},\db^{t,S}_{X \times S})
 \simeq \rh_{\rho_!\D_{\overline{X} \times \overline{S}}}(\rho_!\OO_{\overline{X} \times \overline{S}},\C^{\infty,t,S}_{X \times S}) \\
\OO_{X \times S}^{\mathrm{w},S} \simeq \rh_{\rho_!\D_{\overline{X} \times \overline{S}}}(\rho_!\OO_{\overline{X} \times \overline{S}},\C^{\infty,\mathrm{w},S}_{X \times S}). \end{eqnarray*}
\end{prop}

 Proposition \ref{lem 9} together with Proposition 7.3.2 of \cite{KS3} entail:
\begin{prop} \label{prop 7}Let $G \in D^b_{\rc}(\CC_X)$, $H \in D^b_{\rc}(\CC_S)$. Then
\begin{enumerate}
\item{$\imin\rho\rh(G \boxtimes H,\OO^{t,S}_{X \times S}) \simeq \imin\rho\rh(G \boxtimes \rho_!H,\ot_{X \times S}) \\ \ \ \ \ \simeq \rh(\CC_X \boxtimes H,\tho(G \boxtimes \CC_S,\OO_{X \times S})),$} \\
\item{$\imin\rho\rh(G \boxtimes H,\OO^{\mathrm{w},S}_{X \times S}) \simeq \imin\rho\rh(G \boxtimes \rho_!H,\OW_{X \times S}) \\ \ \ \ \ \simeq \rh(\CC_X \boxtimes H,D'G \boxtimes \CC_S \wtens \OO_{X \times S}).$}
\end{enumerate}
In particular, when $G=\CC_X$ and $H=\CC_S$ we have $\imin\rho\OO^{t,S}_{X \times S} \simeq \OO_{X \times S}$, $\imin\rho\OO^{\mathrm{w},S}_{X \times S} \simeq \OO_{X \times S}$.
\end{prop}
As a consequence of Lemma \ref{lem 7} together with the results in \cite{B} we obtain the following characterization of the sections of $\OO_{X\times S}^{t,S}$:

\begin{prop}
Assume that $U$ (resp. $V$) is a subanalytic Stein open subset of the Stein manifold $X$ (resp. of $S$). Then $R\Gamma(U\times V; \OO^{t,S}_{X\times S})$ is concentrated in degree zero and  $\Gamma(U\times V;\OO^{t;S}_{X\times S})$ is the set of holomorphic functions on $U\times V$ which are tempered on $X\times V$.
\end{prop}
\begin{es}
Let $U=\{z\in\CC, \Im z>0\}$, let $V$ be open subanalytic in $\CC$ and let $g(s)$ be a holomorphic function on $V$. Then, after a choice of a determination of $log\,z$ on $U$, $z^{g(s)}$ defines a section of $\Gamma(U\times V; \OO^{t,S}_{\CC\times \CC})$.
\end{es}

Recall that any distribution on $\R^n$ is, as an hyperfunction, the boundary value of some holomorphic function  on  $\Omega\cap\{(z_1,..., z_n)\in\CC^n, \Im z_i\neq 0\}$, with  moderate growth with respect to $\R^n$, for some  Stein open neighborhood $\Omega$ of $\R^n$ in $\CC^n$. For a precise notion of boundary value and classical hyperfunction theory we refer to the foundational work \cite{SKK}. By $(2)$ of Proposition \ref{prop 6} we deduce the following example:

\begin{es}Let $U=\R_{>0}$ with a coordinate $x$, let $V$ be a subanalytic open set in $\R$ and let $a(s)$ be any continuous function on $V$. Let $f\in\Gamma(\Omega\setminus V;\OO_{\CC})$, where $\Omega$ is an open neighborhood of $V$ in $\CC$, be such that $a$ is the boundary value $vb(f)$ of $f$ as an hyperfunction. Then $x^{a}_{+}:=vb(z^f)$, with $arg \,z\in ]0,2\pi[$,  is a section of $\Gamma(U\times V; \db^{t,S}_{\R\times \R})$.
\end{es}

\begin{oss}\label{controesempio}
As the reader can naturally ask, our method applies only for products of analytic manifolds, i.e, for a projection, since the crucial trick  we used here is that the allowed coverings are formed by products of open subanalytic sets and products are not kept by change of coordinates. So, if we want to treat the case of a general smooth $f:X\to S$,  we have to consider on $X$ a topology with adapted coverings which are less than that of the subanalytic topology. This can be illustrated with  $X=\R^2=\R \times \R$ with coordinates $(x,y)$, $S=\R$ and $f:X\to S$ the second projection. Consider $U  = ]0,1[ \times ]-1,1[$ and the open covering

\begin{eqnarray*}
U_1 & = & U \cap \{y<x\} \\
U_2 & = & ]0,1[ \times ]0,1[
\end{eqnarray*}
(so $U_1$ is not a product of intervals).
Consider the relative tempered distributions $s_1=0$ on $U_1$ and $s_2=\chi_{\{y=2x\}}\exp(1/y)$ on $U_2$ ($\chi_{\{y=2x\}}$ denotes the characteristic function). Then $s_1=s_2=0$ on $U_1 \cap U_2$, hence they glue to a distribution on $U$ which is not relative tempered.

Hence, if we want to realize relative tempered distributions with  respect to a smooth function as a sheaf on a site, we must avoid such kind of coverings. We conjecture, however, that with  a weaker notion of subanalytic site as in \cite{GS} a notion of relative sheaf can be given for a general smooth function.
\end{oss}

\end{document}